\documentclass[final,1p,times,authoryear]{elsarticle}
\usepackage{amsmath}
\usepackage{amsthm}
\usepackage{thmtools}
\usepackage{amssymb}
\usepackage{amsfonts}
\usepackage{mathtools}
\usepackage{hyperref}
\usepackage[ruled,vlined]{algorithm2e}
\usepackage[cmtip,all]{xy}

\newcommand{\abs}[1]{\lvert#1\rvert}

\declaretheoremstyle[notefont=\bfseries,notebraces={}{},
headpunct={},postheadspace=1em]{mystyle}
\declaretheorem[style=mystyle,numbered=no,name=Theorem]{thm}

\newtheorem{theorem}{Theorem}
\newtheorem{lemma}[theorem]{Lemma}
\newtheorem{corollary}[theorem]{Corollary}

\newtheorem{definition}[theorem]{Definition}

\newtheorem{remark}[theorem]{Remark}

\begin{document}

\begin{frontmatter}

\title{Degree bounds for Gr\"{o}bner bases of modules}

\author{Yihui Liang}
\address{Department of Mathematics, Purdue University, 150 N. University Street, West Lafayette, IN47907, USA}
\ead{liang226@purdue.edu}
\ead[url]{www.math.purdue.edu/~liang226/}

\begin{abstract}
	Let $F$ be a non-negatively graded free module over a polynomial ring $\mathbb{K}[x_1,\dots,x_n]$ generated by $m$ basis elements. Let $M$ be a submodule of $F$ generated by elements with degrees bounded by $D$ and dim $F/M=r$. We prove that if $M$ is graded, the degree of the reduced Gr\"{o}bner basis of $M$ for any term order is bounded by $2\left[1/2((Dm)^{n-r}m+D) \right]^{2^{r-1}}$. If $M$ is not graded, the bound is $2\left[1/2((Dm)^{(n-r)^2}m+D) \right]^{2^{r}}$. This is a generalization of Dub\'{e} (1990) and Mayr-Ritscher (2013)'s bounds for ideals in a polynomial ring.
\end{abstract}

\end{frontmatter}

\section{Introduction}
Gr\"{o}bner bases play an important role in computational commutative algebra and algebraic geometry. To analyze the complexity of Gr\"{o}bner bases computations, it is essential to give an upper bound on degrees of polynomials in Gr\"{o}bner bases. Such upper bounds were first proved in characteristic zero, independently by \cite{moller} and \cite{giusti}. In the following results, let $S$ be the polynomial ring $\mathbb{K}[x_1,\dots,x_n]$, $I$ be an ideal  generated by polynomials of degree at most $d$. Giusti proved that if $I$ is a homogeneous ideal in generic coordinates, then the degree of polynomials in the reduced Gr\"{o}bner basis of $I$ with respect to the degree reverse lexicographic order can be bounded by $(2d)^{2^{n-1}}$. Applying results from \cite{bayer} and \cite{giusti}, M\"{o}ller and Mora provided the upper bound $(2d)^{(2n+2)^{n+1}}$ in characteristic zero for any term order. Recently Giusti's bound was improved to $2d^{(n-r)2^{r-1}}$ by \cite{HS}, assuming $I$ has a strongly stable initial ideal and $I$ has dimension $r>1$.

An upper bound in any characteristic was proved by \cite{dube}. Dub\'{e} used a purely combinatorial and constructive argument and proved the upper bound of $2\left(d^2/2+d\right)^{2^{n-1}}$ for any term order. He introduced cone decompositions for ideals (also called Stanley decompositions) and provided algorithms to construct an exact cone decomposition, so that the Gr\"{o}bner basis degree can be bounded by a Macaulay constant associated to the cone decomposition. The Macaulay constant is then bounded by direct computations.

\cite{mayr} incorporated the ideal dimension into Dub\'{e}'s constructions and proved the following dimension-dependent bound. They improved Dub\'{e}'s constructions by using a regular sequence in the ideal which has length equal to the ideal height, and then reducing the problem of bounding the Macaulay constants of an arbitrary ideal to bounding those of a simple monomial ideal generated by pure powers. Based on Mayr-Meyer ideals, Mayr and Ritscher also constructed a family of examples that showed any bound must be at least exponential in the codimension and doubly exponential in the dimension (see\cite[\S 4]{mayr}).

\begin{theorem}[\cite{mayr}] \label{1.2}
	Let $\mathbb{K}$ be an infinite field and $I\subsetneq \mathbb{K}[x_1,\dots,x_n]$ be an ideal of dimension $r$ generated by polynomials $F=\{f_1,\dots,f_s\}$ of degrees $d_1\geq \cdots \geq d_s$. Then for any monomial order, the degree of the reduced Gr\"{o}bner basis $G$ of $I$ is bounded by
	\[
	\text{deg}(G)\leq 2\left[\frac{1}{2}\left((d_1\cdots d_{n-r})^{2(n-r)}+d_1\right)\right]^{2^{r}}.
	\]
	
	If $f_1,\dots,f_s$ are homogeneous, then
	\[
	\text{deg}(G)\leq 2\left[\frac{1}{2}(d_1\cdots d_{n-r}+d_1)\right]^{2^{r-1}}.
	\]
	
\end{theorem}

Despite the fact that Gr\"{o}bner basis computations of modules are the most commonly used ones in computer algebra systems like Macaulay2, CoCoA, and Singular (just to name a few), surprisingly uniform bounds of the above kind are only known for ideals. Generalizing \cite{dube} and \cite{mayr}, we prove the following two bounds on the Gr\"{o}bner basis degrees of graded and non-graded submodules of a free module, respectively. In the following theorems, let $F$ be a free module over $\mathbb{K}[x_1,\dots,x_n]$ with basis elements $e_1,\dots,e_m$ so that deg$(e_j)\geq 0$ for all $j$ and $l=\text{max}\{\text{deg}(e_j):j=1,\dots, m\}$. If $G$ is a Gr\"{o}bner basis, let deg$(G)$ denote the maximum degree of the elements in $G$.

\begin{thm} [37]
	Let $M\subsetneq F$ be a graded submodule generated by homogeneous elements with maximum degree $D\geq l$ and dim$ (F/M)=r$. 
	Then the degree of the reduced Gr\"{o}bner basis $G$ of $M$ for any monomial order on $F$ is bounded by
	\[
	\text{deg}(G)\leq 
	\begin{cases}
	Dmn-n+1, & \text{if } r=0\\
	2\left[\frac{1}{2}((Dm)^{n-r}m+D) \right]^{2^{r-1}} & \text{if } r\geq 1\\
	\end{cases}
	\]
\end{thm}

\begin{thm} [41]
	Let $M\subsetneq F$ be a submodule generated by elements of maximum degree $D\geq l$ and dim$(F/M)=r$. Then the degree of the reduced Gr\"{o}bner basis $G$ of $M$ for any monomial order is bounded by
	\[
	\text{deg} (G) \leq 2\left[\frac{1}{2}\left(\left( Dm  \right)^{(n-r)^{2}}m+D\right) \right]^{2^{r}}.
	\]
\end{thm}

We shall briefly explain how we generalize results of Dub\'{e} and Mayr-Ritscher. Since the initial module is a direct sum of monomial ideals and thanks to the combinatorial nature of Dub\'{e}'s constructions, the algorithms and results of Dub\'{e}'s cone decompositions can be carried through to the module case with minor adjustments. In Section 3-5, we state the necessary statements of cone decompositions for modules, and provide proofs for those that need extra clarifications. To generalize Mayr and Ritscher's proof, we aim to reduce the problem to bounding the Macaulay constants of a simple monomial module which has the form $IF$ with $I$ generated by pure powers. While in the ideal case Mayr and Ritscher uses a regular sequence of length equal to the ideal height, in the module case we find a regular sequence in the zeroth Fitting ideal of the quotient module $F/M$, which has length equal to height of the zeroth Fitting ideal (equivalently height of the annihilator of $F/M$). The reason why we need the zeroth Fitting ideal instead of the annihilator is because the generating degrees of the zeroth Fitting ideal can be bounded linearly by the generating degrees of the module (see Lemma \ref{2.7}). In Section 6-8, we perform the reductions of Mayr and Ritscher to cone decompositions of modules, compute bounds of the Macaulay constants of monomial modules generated by pure powers, and obtain our graded and non-graded Gr\"{o}bner basis bounds. 

\section{Preliminaries}
\subsection{Notation}

The purpose of this subsection is to set up notations that will be used throughout the paper. For a more detailed introduction to Gr\"{o}bner bases and other related topics, the reader can refer to \cite{eisenbud}, \cite{singular}, \cite{kreuzer1}, and \cite{kreuzer2}.

Let $S=\mathbb{K}[x_1,\dots,x_n]$ denote the ring of polynomials in the variables $X=\{x_1,\dots,x_n\}$. Let $F=Se_1\oplus\dots\oplus Se_m$ denote a graded free $S$-module with $l$ = max$\{\text{deg}(e_j):j=1,\dots,m\}$. Without loss of generality, we may assume that all $e_j$ have nonnegative degrees and the one with smallest degree has degree 0. Let Mon$(S)$ and Mon$(F)$ denote the set of monomials of $S$ and $F$, respectively.

Let $\prec_F$ be a monomial order on $F$, whenever there is no confusion we abbreviate $\prec_F$ as $\prec$. Let $f\in F$, then the \textit{initial monomial} of $f$, denoted by in$_\prec(f)$, is the greatest monomial among the monomials of $f$ with respect to $\prec$. If $M$ is a submodule of $F$, then in$_\prec(M)$ is the monomial submodule generated by $\{\text{in}_\prec(f): f\in M \}$. A subset $G=\{g_1,\dots,g_t\}$ in $M$ is called a \textit{Gr\"{o}bner basis} of $M$ if in$_\prec(g_1),\dots,\text{in}_\prec(g_t)$ generates in$_\prec(M)$. Denote in$_{\prec} (G)=\{\text{in}_{\prec}(g):g\in G\}$.

Let $M$ be a submodule of $F$, fix a monomial order $\prec_F$ on $F$ and a Gr\"{o}bner basis $G$ of $M$. For any $f\in F$, let nf$_G(f)$ denote the unique remainder (or normal form) of $f$ with respect to $G$.
Collecting all such remainders of $f\in F$, we denote
\[
N_M=\{\text{nf}_G(f):f\in F\}.
\]

By Macaulay's theorem \cite[Theorem 15.3]{eisenbud}, we have
\[
N_M=\text{span}\{u\in \text{Mon}(F):u\notin \text{in}_{\prec}(M) \}=N_{\text{in}_{\prec}(M)}.
\]
Therefore $N_M$ only depends on in$_\prec(M)$.

Similarly for any ideal $I$ of $S$, any monomial order $\prec_S$ on $S$, and any Gr\"{o}bner basis $G$ of $I$, we denote
\[
N_I=\{\text{nf}_G(p):p\in S \}.
\]

Let $M$ be a submodule of $F$, we denote $\text{dim}(F/M)=\text{dim}(S/\text{ann}_S(F/M))$ to be the Krull dimension of the $S$-module $F/M$. Recall that ann$_S(F/M)=M:_S F:=\{p\in S : pF\subset M \}$. If $I$ is an ideal of $S$, let ht$(I)$ denote the height of $I$. If $T$ is a vector space over $\mathbb{K}$, let dim$_{\mathbb{K}}(T)$ denote the vector space dimension of $T$.

\subsection{Hilbert function}

Let $T\subseteq F$ be a graded $\mathbb{K}$-vector space with graded components
\[
T_z=\{f\in T: f \text{ is homogeneous of degree } z\}\cup \{0\},
\]
then the \textit{Hilbert function} of $T$ is defined as 
\[
\text{HF}_T(z)=\text{dim}_{\mathbb{K}}(T_z).
\]
The \textit{Hilbert series} of $T$ is defined as 
\[
\text{HS}_T(t)=\sum_{z\geq 0} \text{HF}_T(z)t^z.
\]

Let $M$ be a submodule of $F$, there exists a unique polynomial which is equal to HF$_{N_M}(z)$ for sufficiently large $z$. This polynomial is called the \textit{Hilbert polynomial} of $N_M$ and will be denoted as HP$_{N_M}(z)$. The \textit{Hilbert regularity} of $N_M$ is defined as min$\{z_0 \in \mathbb{Z}:\text{HF}_{N_M}(z)=\text{HP}_{N_M}(z) \; \forall z\geq z_0 \}$.

Notice that $F/M$ and $N_M$ have the same Hilbert function, hence dim$(F/M)$ = deg(HP$_{N_M})+1$ (with the convention that deg$(0)=-1$). 

\subsection{Regular sequence}

Recall that a sequence of elements $g_1,\dots,g_k$ in $S$ is called a \textit{regular sequence} if
\begin{enumerate}
	\item $g_i$ is a non-zerodivisor on $S/(g_1,\dots,g_{i-1})$ for all $i=1,\dots,k$ and
	\item $(g_1,\dots,g_k)\neq S $.
\end{enumerate}

One of the many nice properties that homogeneous regular sequences have is that the submodules they generate have the same Hilbert functions if they have the same degrees. 
\begin{lemma} \label{2.1}
	Let $J=(g_1,\dots,g_k)$ be an ideal generated by a homogeneous regular sequence in $S$ with degrees $d_1,\dots,d_k$. Fix an arbitrary monomial ordering on $F$, consider $JF\subset F$, then $F/JF$ or equivalently $N_{JF}$ has the Hilbert series
	\[
	\text{HS}_{N_{JF}}(t)=\frac{(\sum_{i=1}^{m}t^{\text{deg}(e_i)})(\prod_{i=1}^{k}(1-t^{d_i}))}{(1-t)^{n}}.
	\]
	The Hilbert regularity of $N_{JF}$ is $d_1+\cdots+d_k+l-n+1$.
\end{lemma}

\begin{proof}
	See \cite[Corollary 5.2.17]{kreuzer2}.
\end{proof}

It is a well-known fact that given a polynomial ideal of dimension $r$ in $n$ variables over an infinite field, we may assume $n-r$ generators of $I$ form a regular sequence. This will be one of the key constructions to achieve a dimension-dependent bound.
\begin{lemma} \label{2.2}
	Let $\mathbb{K}$ be an infinite field and $I\subsetneq S$ an ideal generated by homogeneous polynomials $p_1,\dots,p_s$ with degrees $d_1\geq \cdots \geq d_s$ such that dim$(S/I) \leq r$. Then there are a strictly decreasing sequence $s\geq j_1 > \cdots > j_{n-r}\geq 1$ and homogeneous $a_{ki}\in S$ such that
	\[
	g_k=\sum_{i=j_k}^{s} a_{ki}p_{i} \qquad \text{for } k=1,\dots,n-r
	\]
	form a homogeneous regular sequence, dim$(S/(p_{j_k},\dots,p_s))=n-k$, and deg$(g_k)=d_{j_k}$.
\end{lemma}

\begin{proof}
	See \cite[Proposition 4.17]{binaei}
\end{proof}

\begin{remark} \label{2.3}
	For simplicity we will use the weaker version of the above lemma, that is we may assume the regular sequence $g_1,\dots,g_{n-r}$ have (the largest $n-r$) degrees $d_{n-r},\dots,d_1$. So one could potentially get a better bound in Theorem \ref{7.7} by using the degrees $d_{j_k}$ in Lemma \ref{2.2}.
	
	Note that we can reduce to the weaker version since for each $i$, there exists a regular element $y_i\in S_1$ on $S/(g_1,\dots,\hat{g_i},\dots,g_{n-r})$, so we can replace $g_i$ by $g_i y_i^{d_{n-r-i+1}-\text{deg}(g_i)}$.
\end{remark}

\subsection{Fitting ideal}

\begin{definition}
	Let $M$ be an $S$-module and choose a presentation
	\[
	S^s\;{\xrightarrow {\ \varphi\ }}\;S^t \;{\xrightarrow {\ \ }}\;M\;{\xrightarrow {\ \ }}\; 0.
	\]
	of $M$. Let $I_k(\varphi)$ denote the $S$-ideal generated by all $k\times k$ minors of $\varphi$ (set $I_k(\varphi)=S$ for $k\leq 0$ and $I_k(\varphi)=0$ for $k>min\{s,t\}$). Then the \textit{$i$th Fitting ideal} of $M$ is defined as Fitt$_i(M)=I_{t-i}(\varphi)$, which is independent of the choice of the presentation (see \cite[Corollary-Definition 20.4]{eisenbud}).
\end{definition}

If $I$ is an ideal in $S$, let 
\[
\sqrt{I}=\{x\in S: x^k\in I \text{ for some } k\in \mathbb{N} \}
\]
be the radical of $I$, then dim$(S/I)$ = dim$(S/\sqrt{I})$. 

\begin{lemma}
	Let $M$ be an $S$-module, then the $0$th Fitting ideal Fitt$_0(M)$ satisfies Fitt$_0(M)\subseteq$ ann$_S(M)$ and $\sqrt{\text{Fitt}_0(M)}=\sqrt{\text{ann}_S(M)}$.
\end{lemma}

\begin{proof}
	See \cite[Proposition 20.7]{eisenbud}.
\end{proof}

Applying Remark \ref{2.3} to the $0$th Fitting ideal of $F/M$, we have
\begin{lemma} \label{2.6}
	Let $M$ be a graded submodule of $F$ with dim$(F/M)$ = dim$(S/M:_S F)=r$. Let Fitt$_0(F/M)$ be the $0$th Fitting ideal of $F/M$ generated by polynomials $p_1,\dots,p_k$ of degrees $d_1\geq \cdots \geq d_k$. Then Fitt$_0(F/M)\subseteq M:_S F$ contains a regular sequence $g_1,\dots,g_{n-r}$ of degrees $d_1\geq \cdots \geq d_{n-r}$.
\end{lemma}

\begin{proof}
	Since $\sqrt{M:_S F}=\sqrt{\text{Fitt}_0(F/M)}$, we have $r=$ dim$(S/M:_S F)=$\linebreak dim$(S/$Fitt$_0(F/M))$. Now apply Remark \ref{2.3} to Fitt$_0(F/M)$.
\end{proof}

Notice that we need the Fitting ideal Fitt$_0(F/M)$ to replace the annihilator $M:_S F$ because the generating degrees of Fitt$_0(F/M)$ can be bounded linearly by the generating degrees of $M$, while the generating degrees of $M:_S F$ are usually much larger. We will need both the graded and non-graded versions of the following lemma.
\begin{lemma} \label{2.7}
	Let $M$ be a submodule of $F$ generated by (homogeneous) elements $f_1,\dots,f_s$ with degrees $D=D_1\geq \cdots \geq D_s$ and $s\geq m$. If Fitt$_0(F/M)\neq 0$, then the maximum degree of a minimal (homogeneous) generating set of Fitt$_0(F/M)$ is bounded by $D_1+\cdots +D_m-\sum_{j=1}^{m} \text{deg}(e_j) \leq Dm$.
\end{lemma}

\begin{proof}
	For each $i=1,\dots,s$, we can write $f_i=a_{1i}e_1+\cdots+a_{mi}e_m$ for some $a_{ji}\in S$, deg$(a_{ji})\leq D_i-\text{deg}(e_j)$ (or deg$(a_{ji})= D_i-\text{deg}(e_j)$ if we are in the homogeneous setting and $a_{ji}\neq 0$).
	Consider the presentation
	\[
	S^s\;{\xrightarrow {\ \varphi\ }}\;S^m \cong F\;{\xrightarrow {\ \ }}\;F/M\;{\xrightarrow {\ \ }}\; 0.
	\]
	of $F/M$ where $\varphi=(a_{ji})_{j=1,\dots,m,i=1,\dots,s}$. Then Fitt$_0(F/M)=I_m(\varphi)$ is the ideal generated by all the $m\times m$ minors of $\varphi$. Every such minor corresponds to $m$ elements $f_{i_1},\dots,f_{i_m}$ among the generators of $M$, so if the minor is nonzero, it will be a (homogeneous) element of degree $\leq D_{i_1}+\cdots +D_{i_m}-\sum_{j=1}^{m} \text{deg}(e_j)$. Hence the largest degree of a minimal (homogeneous) generating set of Fitt$_0(F/M)$ is bounded by $D_1+\cdots +D_m-\sum_{j=1}^{m} \text{deg}(e_j) \leq Dm$.
\end{proof}

\subsection{Homogenization}
Let $f\in F$ be a nonzero element with homogeneous components $f_j$. Let $t$ be a new variable, then the \textit{homogenization} of $f$ is defined as $f^h=\sum_{j} f_j t^{\text{deg}(f)-\text{deg}(f_j)}$.
For a module $M\subseteq F$, the \textit{homogenization} of $M$ is the submodule generated by $\{f^h:f\in M\}$, denoted as $M^h=\left<f^h:f\in M\right>$.

If $f\in F^h=S[t]e_1 \oplus \cdots \oplus S[t]e_m$, let $f^{deh}$ denote the \textit{dehomogenization} of $f$ which is obtained from $f$ by substituting $t$ by $1$. If $G$ is a subset of $F^h$, let $G^{deh}=\{f^{deh}:f\in G \}$.

 Given a monomial order $\prec$ on $F$, we extend it to a monomial order $\prec^\prime$ on $F^h$ as follows:
\[
\begin{split}
\boldsymbol{x^a}t^ce_i \prec^\prime \boldsymbol{x^b}t^de_j \Longleftrightarrow \;
&\text{deg}(\boldsymbol{x^a}e_i)+c<\text{deg}(\boldsymbol{x^b}e_j)+d,\\ 
& \text{ or deg}(\boldsymbol{x^a}e_i)+c=\text{deg}(\boldsymbol{x^b}e_j)+d \text{ and }\boldsymbol{x^a}e_i \prec \boldsymbol{x^b}e_j.
\end{split}
\]
It is easy to check that $\text{in}_{\prec}(f^{deh})=\text{in}_{\prec^\prime}(f)^{deh}$ for all homogeneous $f\in F^h$.

Using the above extension, we can obtain a Gr\"{o}bner basis of a non-graded module $M$ from dehomogenizing a homogeneous Gr\"{o}bner basis of any graded module that dehomogenizes into $M$.
\begin{lemma} \label{2.9}
	Let $M=Sf_1+\cdots+Sf_s$ be a submodule of $F$ and $\prec$ be a monomial order on $F$. If $N$ is a graded submodule of $F^h$ where $S[t]f_1^h+\cdots+S[t]f_{s}^h\subseteq N \subseteq M^h$ and $G$ is a homogeneous Gr\"{o}bner basis of $N$ with respect to $\prec^\prime$, then $G^{deh}$ is a Gr\"{o}bner basis of $M$ with respect to $\prec$.
\end{lemma}

\begin{proof}
	Notice that as $S[t]f_1^h+\cdots+S[t]f_s^h \subseteq N\subseteq M^h$, we have $N^{deh}=Sf_1+\cdots+Sf_s=M$. Since  in$_{\prec}(f^{deh})=$ in$_{\prec^\prime}(f)^{deh}$ for all homogeneous $f\in N$, we have in$_{\prec}(N^{deh})=$ in$_{\prec^\prime}(N)^{deh}$.
	Hence in$_{\prec}(M)=$ in$_{\prec}(N^{deh})=$ in$_{\prec^\prime}(N)^{deh}=\left<\text{in}_{\prec^\prime}(G)^{deh}\right>=\left<\text{in}_{\prec}(G^{deh})\right>$.
\end{proof}

\section{Cone decomposition}
\cite{dube} introduced cone decompositions to express a subspace $T$ in $S$ as a finite direct sum of subspaces of the form $h\mathbb{K}[u]$, where $h$ is a polynomial in $S$ and $u$ is a subset of $X$. In this section, we will give an introduction to cone decompositions in the context of free modules.
\begin{definition}
	Let $h$ be a homogeneous element in $F$ and $u$ a subset of $X$, then $C=C(h,u)=h\mathbb{K}[u]$ is called a \textit{cone}. The degree of the cone is defined as deg$(C)=$ deg$(h)$, the dimension of the cone is defined as dim$(C)=\abs{u}$, and $h$ is called the pivot of the cone.
\end{definition}

\begin{definition}
	Let $T\subseteq F$, $h_1,\dots,h_r$ be homogeneous elements in $F$, and  $u_1,\dots,$\linebreak $u_r$ be subsets of $X$. If as $\mathbb{K}$-vector spaces,
	\[
	T = \bigoplus_{i=1}^{r} h_i\mathbb{K}[u_i],
	\]
	then $P=\{C(h_1,u_1),\dots,C(h_r,u_r)\}$ is called a \textit{cone decomposition} of $T$. The degree of the cone decomposition is defined as deg$(P)=$ max$\{\text{deg}(C):C\in P\}$.
\end{definition}

If $T$ admits a cone decomposition $P$, then the Hilbert function of $T$ is a sum of the Hilbert functions of $h\mathbb{K}[u]$. Counting the number of monomials in $h\mathbb{K}[u]$, we get that if $u=\emptyset$, then
\[
\text{HF}_{h\mathbb{K}[\emptyset]}(z) =
\begin{dcases}
0, & \text{if } z\neq \text{deg}(h) \\
1, & \text{if } z=\text{deg}(h) \\
\end{dcases}
\]
and if $\abs{u}>0$,
\[
\text{HF}_{h\mathbb{K}[u]}(z) =
\begin{dcases}
0, & \text{if } z< \text{deg}(h) \\
\binom{z-\text{deg}(h)+\abs{u}-1}{\abs{u}-1}, & \text{if } z\geq \text{deg}(h) \\
\end{dcases}
\]

Therefore cones of the form $h\mathbb{K}[\emptyset]$ only contribute to finitely many values of the Hilbert function HF$_T$. Collecting the remaining cones, we denote
\[
P^{+}:=\{C\in P : \text{dim}(C)>0\}
\]

\begin{definition}
	A cone decomposition $P$ for $T$ is said to be \textit{q-standard} if the following two conditions hold:
	\begin{enumerate}
		\item There is no cone $C\in P^{+}$ with deg$(C)<q$.
		\item For every $C\in P^{+}$ and degree d such that $q\leq d\leq \text{deg}(C)$, $P$ contains a cone $C^\prime$ with deg$(C^\prime)=d$ and dim$(C^\prime) \geq $ dim$(C)$.
	\end{enumerate}
\end{definition}

Notice that if $P^+=\emptyset$, then $P$ is $q$-standard for all natural numbers $q$. If $P^{+}\neq \emptyset$, then the only possible value for $q$ is min$\{\text{deg}(C) : C\in P^{+}\}$.

We now define a special cone decomposition that splits a cone and is useful for manipulating cone decompositions.
\begin{definition} \label{3.4}
	Let $u=\{x_{j_1},\dots,x_{j_t}\}\subseteq X$, $h$ a homogeneous element in $F$, and $C=h\mathbb{K}[u]$. Then the \textit{fan} of $C$ is defined as
	\[
	\textbf{F}(C):=\{\{C(h,\emptyset)\}\cup\{C(x_{j_i}h,\{x_{j_1},\dots,x_{j_i}\}):i=1,\dots,t \}\}.
	\] 
\end{definition}

\begin{remark}
	The following list contains some facts of $q$-standard cone decompositions that are easy to verify.
	\begin{enumerate}
		\item $\{C(e_j,X)\}$ is a deg$(e_j)$-standard cone decomposition of $Se_j$.
		\item \label{3.5(2)}
		If $C$ is a cone, $\textbf{F}(C)$ is a $(\text{deg}(C)+1)$-standard cone decomposition of $C$.
		\item \label{3.5(3)}
		Let $T=T_1 \oplus T_2$ and let $P_1$ and $P_2$ be $q$-standard cone decompositions of $T_1$ and $T_2$, respectively. Then $P_1 \cup P_2$ is a $q$-standard cone decompositions of $T$ with deg$(P_1\cup P_2)\geq$  deg$(P_1)$.
		\item \label{3.5(4)}
		Let $T$ be a subset of $S$. If $P=\{C(p_1,u_1),\dots,C(p_r,u_r)\}$ is a $q$-standard cone decomposition of $T$, then for any homogeneous element $f\in F$, the set $Q=fP=\{C(fp_1,u_1),\dots,C(fp_r,u_r)\}$ is a $(q+\text{deg}(f))$-standard cone decomposition of $fT=\{fp:p\in T\}$.
	\end{enumerate}
\end{remark}

\begin{lemma} \label{3.6}
	Let $P$ be a $q$-standard cone decomposition of a subspace $T\subseteq F$. Then for any $d \geq q$, there exists a $d$-standard cone decomposition $P_d$ of $T$ with $deg(P_d)\geq$ deg$(P)$ and deg$(P_d^+)\geq$ deg$(P^+)$.
\end{lemma}

\begin{proof}
	If $P^{+}=\emptyset$ then the result holds trivially. Assume $P^{+}\neq\emptyset$, it suffices to show that there exists a $(q+1)$-standard cone decomposition $P_{q+1}$ of $T$ with deg$(P_{q+1})\geq$ deg$(P)$ and deg$(P_{q+1}^+)\geq$ deg$(P^+)$. Let $Q=\{C \in P : \text{deg}(C)=q\}$. Notice that $Q$ is trivially $q$-standard and $P\setminus Q$ is $(q+1)$-standard as $P$ is $q$-standard, 
	
	By Remark \hyperref[3.5(2)]{3.5(2)}, for each $C \in Q$, there exists a $(q+1)$-standard cone decomposition $\textbf{F}(C)$ of $C$. Now apply Remark \hyperref[3.5(3)]{3.5(3)} to see that $P_{q+1}:=\bigcup_{C \in P} \textbf{F}(C) \cup (P\setminus Q)$ is a $(q+1)$-standard cone decomposition of $T$. Also deg$(P_{q+1})\geq$ deg$(P)$ and deg$(P_{q+1}^+)\geq$ deg$(P^+)$ are clear from the definition of $\textbf{F}(C)$.
\end{proof}

\section{Decomposing a set of normal forms}

For an ideal $I$, \cite{dube} constructed a $0$-standard cone decomposition $P$ of $N_I$, so that the degree of $P$ gives an upper bound to the Gr\"{o}bner basis degree of $I$. We will follow Dub\'{e} to construct an $l$-standard cone decomposition $Q$ of $N_M$ with deg$(Q)$ bounding the Gr\"{o}bner basis degree of $M$. Most of the statements in Dub\'{e} holds in the module case and the proofs can be applied directly with some slight modifications. Recall that $N_M=N_{in_{\prec}(M)}$, hence it suffices to decompose $N_M$ for monomial modules. 

\begin{definition}
	Let $P\cup Q$ be a cone decomposition of $T\subseteq F$, and let $M$ be a submodule of $F$. Then $P$ and $Q$ are said to \textit{split $T$ relative to $M$} if $C\in P$ implies $C\subseteq M$, and $C\in Q$ implies $C\cap M = {0}$.
\end{definition}

\begin{lemma} \label{4.2}
	Let $P=\{C(g_1,u_1),\dots,C(g_r,u_r)\}$ and $Q=\{C(h_1,v_1),\dots,C(h_s,v_s)\}$ split $T$ relative to a monomial module $M$, where for each $C(h_i,v_i)\in Q$, $h_i$ is a monomial in $F$. Then $P$ is a cone decomposition for $T\cap M$ and $Q$ is a cone decomposition for $T\cap N_M$.
\end{lemma}

\begin{proof}
	See \citet[Lemma 4.1]{dube}.
\end{proof}

Using Dub\'{e}'s SPLIT algorithm \cite[\S 4]{dube}, we can produce cone decompositions $P$ and $Q$ which split a cone $h\mathbb{K}[u]$ relative to a monomial module $M$. See \citet[Lemma 4.3, 4.4]{dube} for the termination and correctness of the algorithm \hyperref[SPLIT]{SPLIT}.

\begin{algorithm}[h] \label{SPLIT}
	\KwIn{$h\in \text{Mon}(F)$, $u\subseteq X$ a set of variables, $M$ a monomial submodule of $F$, and $B$ a monomial generating set of $M:_S h$}
	\KwOut{cone decompositions $(P,Q)$ which splits $h\mathbb{K}[u]$ relative to $M$}
	\If{$1\in B$}
	{\Return{$(P=\{C(h,u)\},Q=\emptyset)$}}
	\If{$B\cap \text{Mon}(\mathbb{K}[u])=\emptyset$}
	{\Return{$(P=\emptyset,Q=\{C(h,u)\})$}}
	\Other
	{
		Choose $s\subset u$ a maximal subset such that $B\cap \text{Mon}(\mathbb{K}[s])=\emptyset$
		
		Choose $x_i\in u\setminus s$  \tcp*[f]{If $s=u$ this point would not be reached}
		
		$(P_0,Q_0):=$SPLIT$(h,u\setminus \{x_i\},M,B)$
		
		$B^\prime := \{x_i^{-1}f : f\in B, x_i \text{ divides } f \}\cup \{f\in B :  x_i\text{ does not divide } f \}$
		
		$(P_1,Q_1):=$SPLIT$(x_ih,u,M,B^\prime)$
		
		\Return{$(P=P_0\cup P_1, Q=Q_0\cup Q_1)$}
	}
	\caption{{\bf SPLIT$(h,u,M, B)$}}
\end{algorithm}

The cone decomposition $Q$ produced by the SPLIT algorithm has the crucial property that its cone decomposition degree bounds the Gr\"{o}bner basis degree of $M$ (see Theorem \ref{4.7}). To see this, we start with the following lemma.
\begin{lemma} \label{4.5}
	Let $M=\bigoplus^{m}_{j=1} {I_j}e_j$ where $I_j$ are monomial ideals. For each $j$, let $B_j$ be a monomial generating set of ${I_j}$ and $(P_j,Q_j)=$ SPLIT$(e_j,X,I_je_j,B_j)$. If $I_j\neq S$, then ${I_j}e_j$ can be generated by the set $\{f \in B_je_j : \text{deg}(f) \leq 1+\text{deg}(Q_j)\}$. Hence $M$ can be generated by the set $\bigcup^{m}_{j=1} \{f \in B_je_j : \text{deg}(f) \leq \text{max}\{1+\text{deg}(Q_j),\text{deg}(e_j)\} \}$.
\end{lemma} 

\begin{proof}
	We sketch the proof given in \citet[Lemma 4.8]{dube}. It suffices to show that for every minimal monomial generator $f$ of $I_je_j$, $Q_j$ contains a cone $C(h,u)$ with deg$(h)=\text{deg}(f)-1$. Since $P_j$ is a cone decomposition of $I_je_j$ and $f$ is a minimal generator, $P_j$ must contain a cone of the form $C(f,v)$. Then a cone of the form $C(h,u)\in Q_j$ with deg$(h)=\text{deg}(f)-1$ could be found by tracing back and forth the recursions.
\end{proof}

\begin{lemma} \label{4.6}
	Let $(P,Q)=	\text{SPLIT}(h,u,M,B)$, then $Q$ is a deg$(h)$-standard cone decomposition.
\end{lemma}

\begin{proof}
	See \citet[4.10]{dube}.
\end{proof}

Combining Corollary \ref{4.5} and Lemma \ref{4.6}, we obtain our main theorem of this section. 
\begin{theorem} \label{4.7}
	Let $G$ be a homogeneous Gr\"{o}bner basis of a graded submodule $M \subseteq F$ with respect to a monomial order $\prec_F$. Then $N_M$ admits a $l$-standard cone decomposition $Q$ where $l = \text{max}\{\text{deg}(e_1),\dots, \text{deg}(e_m)\}$ and $G^\prime = \{g \in G : \text{deg}(g)\leq \text{max}\{1+\text{deg}(Q),l\}\}$ is also a Gr\"{o}bner basis of M with respect to $\prec_F$. In particular, the degree of the reduced Gr\"{o}bner basis of $M$ is bounded by $\text{max}\{1+\text{deg}(Q),l\}$.
\end{theorem}

\begin{proof}
	Since in$_\prec (M)$ is a monomial submodule, in$_\prec (M) = \bigoplus^{m}_{j=1} {I_j}e_j$ for some monomial ideals $I_j \subseteq S$. For each $j=1, \dots, m$, let $B_je_j =$ in$_\prec(G)\cap Se_j$, then $B_je_j$ generates ${I_j}e_j $. Consider $(P_j,Q_j) =$ SPLIT$(e_j, X, I_je_j, B_j)$. $Q_j$ is a deg($e_j$)-standard cone decomposition of $N_{{I_j}e_j} \cap Se_j$ by Lemma \ref{4.2} and Lemma \ref{4.6}. Then by Lemma \ref{3.6}, there exists a $l$-standard cone decomposition $Q_{j}^\prime$ for $N_{{I_j}e_j} \cap Se_j$ with deg$(Q_j^{\prime})\geq \text{deg}(Q_j)$.
	
	Since $N_M=N_{\text{in}_{\prec}(M)}= \bigoplus^{m}_{j=1} N_{{I_j}e_j} \cap Se_j$, we have that $Q:=\bigcup^{m}_{j=1} Q_{j}^\prime $ is a $l$-standard cone decomposition for $N_M$ and clearly deg$(Q) \geq \text{deg}(Q_{j}^\prime)$. By Corollary \ref{4.5}, in$_{\prec}(M)$ can be generated by 	
	\[
	\begin{split}
	&\;\;\;\; \bigcup^{m}_{j=1} \{f \in B_je_j : \text{deg}(f) \leq \text{max}\{1+\text{deg}(Q_j),\text{deg}(e_j)\} \} \\
	&\subseteq\bigcup^{m}_{j=1} \{f \in B_je_j : \text{deg}(f) \leq \text{max}\{1+\text{deg}(Q_j^{\prime}),\text{deg}(e_j)\} \}  \\
	&\subseteq\{f \in \bigcup^{m}_{j=1} B_je_j : \text{deg}(f) \leq \text{max}\{1+\text{deg}(Q),l\} \}  \\
	&= \{\text{in}_{\prec}(g) : g \in G^\prime\}.\\
	\end{split}
	\]
	
	Hence $G^\prime$ is a Gr\"{o}bner basis of $M$.	
\end{proof}

\section{The exact cone decomposition and Macaulay constants}

One of the nice properties that a $q$-standard cone decomposition $P$ has is that there are cones of every degree between $q$ and deg$(P)$. However, this doesn't give us any control over the number of cones in a certain degree. Therefore the following notion is introduced to further refine a $q$-standard cone decomposition.
\begin{definition}
	Let $T$ be a subspace of $F$, then $P$ is called a $q$-\textit{exact} cone decomposition of $T$ if $P$ is a $q$-standard cone decomposition of $T$, and deg$(C)\neq \text{deg}(C^{\prime})$ for all $C\neq C^{\prime}\in P^+$
\end{definition}

If $P^{+}\neq \emptyset$, then there is a unique value $q>0$ such that $P$ is $q$-standard. If $P^{+}=\emptyset$, $P$ is trivially $q$-exact for all natural numbers $q$, in this case we set $P$ to be $0$-exact for the following definition.

\begin{definition}
	Let $P$ be a $q$-exact cone decomposition of $T\subseteq F$.
	Then the \textit{Macaulay constants} of $P$ are defined as
	\[
	b_k:=\text{max}\left(\{q\}\cup \{1+\text{deg}(C):C\in P, \text{dim}(C)\geq k\}\right) \quad \text{for } k=0,\dots,n+1.
	\]
\end{definition}

It is a simple consequence of this definition that the $b_k$'s satisfy $b_0\geq b_1\geq \dots \geq b_{n+1}=q$, and if $P^+\neq \emptyset$ then $b_0=1+\text{deg}(P)$ and $b_1=1+\text{deg}(P^+)$.

Once we have an exact cone decomposition $P$, the Macaulay constants give a nearly complete picture of $P$, meaning that they control degrees of all the cones in $P$ but not the specific pivots.
\begin{lemma} \label{5.3}
	Let $P$ be a $q$-exact cone decomposition, and let $b_0,\dots,b_{n+1}$ be defined as above. Then for each $i=1,\dots,n$ and degree $d$ such that $b_{i+1}\leq d < b_i$, there is exactly one cone $C\in P^{+}$ such that deg$(C)=d$, and for that cone dim$(C)=i$. In particular $b_i=b_{i+1}+\abs{\{C\in P^+ : \text{dim}(C)=i\}}$ for $i=1,\dots,n$.
\end{lemma}

\begin{proof}
	See \citet[Lemma 6.1]{dube}.
\end{proof}

One way to make a $q$-standard cone decomposition $q$-exact is as follows: whenever there are two cones of the same degree, replace the cone of lower-dimension by its fan and the resulting cone decomposition still remains $q$-standard. The \hyperref[EXACT]{EXACT} algorithm is from \cite{mayr} and is a reformulation of SHIFT and EXACT in \cite{dube}.
\begin{algorithm}[h] \label{EXACT}
	\KwIn{$Q$ a $q$-standard cone decomposition of $T\subseteq F$}
	\KwOut{$P$ a $q$-exact cone decomposition of $T\subseteq F$}
	$P:=Q$
	
	\For{$d:=q,\dots,\text{deg}(P^+)$}
	{
		$S:=\{C\in P^+ : \text{deg}(C)=d\}$
		
		\While{$\abs{S}>1$}
		{
			Choose $C\in S$ with minimal dimension $\text{dim}(C)$
			
			$S:=S\setminus \{C\}$
			
			$P := P\setminus \{C\} \cup \textbf{F}(C)$
		}
	}
	\Return $(P)$
	\caption{{\bf EXACT$(Q)$}}
\end{algorithm}

The EXACT algorithm results in the following lemma.
\begin{lemma} \label{5.4}
	Every $q$-standard cone decomposition $Q$ of a vector space $T\subseteq F$ may be refined into a $q$-exact cone decomposition $P$ of $T$ with deg$Q)\leq \text{deg}(P)$ and deg$(Q^+)\leq \text{deg}(P^+)$.
\end{lemma}

\begin{proof}
	See \citet[Lemma 16]{mayr}.
\end{proof}

As we have seen before, if $T$ admits a cone decomposition $P$, then its Hilbert polynomial is determined by $P$. If we further assume $P$ is exact, then by Lemma \ref{5.3} $P$ is determined by the Macaulay constants of $P$. Hence it follows that the Hilbert polynomial HP$_T$ is determined by the Macaulay constants.
\begin{lemma} \label{5.5}
	Let $P$ be a $q$-exact cone decomposition of a subspace $T\subseteq F$, let $b_0,\dots,b_{n+1}$ be the Macaulay constants of $P$. Then for $z\geq b_0$, the Hilbert function HP$_T(z)$ attains the polynomial form
	\begin{equation*}
	\text{HP}_T(z)=\binom{z-b_{n+1}+n}{n}-1-\sum_{i=1}^{n} \binom{z-b_i+i-1}{i}.
	\end{equation*}
	In addition for $z \geq b_1$, 
	\begin{equation*}
	\text{HF}_T(z)=\text{HP}_T(z)+\abs{\{C(h,\emptyset)\in P : \text{deg}(h)=z\}}
	\end{equation*}
\end{lemma}

\begin{proof}
	See \citet[\S 7]{dube}.
\end{proof}

The converse of the previous lemma also holds, that is the Macaulay constants are determined by the Hilbert polynomial. In particular the Macaulay constants does not depend on the chosen $q$-exact cone decomposition. 
\begin{lemma} \label{5.6}
	Let $P$ be any $q$-exact cone decomposition for a subspace $T\subseteq F$. Then the Macaulay constants $b_1,\dots,b_{n+1}$ are uniquely determined by HP$_T$ and $q$, and $b_0=\text{min}\{d\geq b_1:\text{HP}_T(z)=\text{HF}_T(z) \; \forall z\geq d \}$.
\end{lemma}

\begin{proof}
	See \citet[Lemma 7.1]{dube}.
\end{proof}

The above lemma also shows that if $T$ is a subspace with a known Hilbert regularity (e.g. set of normal forms of a submodule generated by a homogeneous regular sequence), then to bound $b_0$ it suffices to bound $b_1$.

\section{Reduction to the complete intersection case}
So far we have reduced our problem of bounding the Gr\"{o}bner basis degree into bounding the Macaulay constant $b_0$. However as in Section 8 of \citet{dube}, attacking this problem directly requires a large amount of computations because our module is arbitrary. \citet{mayr} used Lemma \ref{2.2} to reduce the problem into bounding the Macaulay constant of a complete intersection, thereby simplifying the computations and improving the bound for ideals of small dimension. Adopting their approach, in this section we show that if $IF\subseteq M$ with $I$ a complete intersection (which exists by Lemma \ref{2.6} if $\mathbb{K}$ is infinite), then the Macaulay constant of $N_{IF}$ bounds the Macaulay constant of $N_M$. By Lemma \ref{2.1} and Lemma \ref{5.6}, if $I$ is generated by polynomials of degrees $d_1,\dots,d_{n-r}$, it suffices to bound the Macaulay constant $b_1$ of $N_{JF}$ where $J=(x_1^{d_1},\dots,x_{n-r}^{d_{n-r}})$.

Notice that Lemma \ref{6.2} and \ref{6.3} require a monomial order defined on $S$, but its sole purpose is to define $N_I$ for an ideal $I$ in $S$ and is irrelevant to the main theorems.

\begin{lemma} \label{6.2}
	Let $M$ be a submodule of $F$ generated by elements $g_1,\dots,g_t,f_1\dots,f_s \in F$ and let $L=Sg_1+\dots+Sg_t\subseteq M$. Then for a fixed monomial order on $S$,
	\[
	M=L\oplus \bigoplus_{i=1}^{s} f_i \cdot N_{L_{i-1}: f_i}
	\]
	where $L_{k}=(g_1,\dots,g_t,f_1,\dots,f_k)$ for $k=0,\dots,s$.
\end{lemma}

\begin{proof}
	Apply the construction in \citet[\S 2.2 Example 2]{dube} to get $L+Sf_1=L\oplus f_1N_{L:f_1}$ and proceed inductively.
\end{proof}

In order to reduce to $IF\subseteq M$ with $I$ a complete intersection, we show that for any submodule $L\subseteq M$, the Macaulay constant $b_0$ (equivalently the degree of cone decomposition) does not decrease if we replace $N_M$ by a vector space $T$ whose Hilbert function equal to the Hilbert function of $N_{L}$. 
\begin{lemma} \label{6.3}
	Let $M$ be a graded submodule of $F$ generated by homogeneous elements $g_1,\dots,g_t,$\linebreak$f_1\dots,f_s\in F$, and fix a monomial order $\prec_S$ on $S$ and a monomial order $\prec_F$ on $F$. let $L=Sg_1+\dots+Sg_t\subseteq M$ and $D=\text{max}\{\text{deg}(f_i):i=1,\dots,s\}\geq l$. Then if $Q$ is an $l$-standard cone decomposition of $N_M$, then there exists a vector space $T\subseteq F$ and a $D$-exact cone decomposition $P$ of $T$ such that $\text{HF}_T=\text{HF}_{N_L}$ and $\text{deg}(Q)\leq \text{deg}(P)$.
\end{lemma}

\begin{proof}
	We use the notation defined in the previous lemma. By \citet[Theorem 4.11]{dube}, there exists $0$-standarad cone decompositions $Q_k$ of $N_{L_{k-1}:f_k}$ for each $k=0,\dots,s$. Hence by Remark \hyperref[3.5(4)]{22(4)} $f_kQ_k$ is a deg$(f_k)$-standard cone decompositions of $f_kN_{L_{k-1}:f_k}$. Then by Lemma \ref{3.6}, $Q,f_1Q_1,\dots,f_sQ_s$ can be converted into $D$-standard cone decompositions $\tilde Q,\tilde Q_1, \dots,\tilde Q_s$. Define $T:=\oplus_{i=1}^{s} f_iN_{L_{i-1}:f_i} \oplus N_M$, then we have
	\[
	F=M \oplus N_M=L \oplus \bigoplus_{i=1}^{s} f_iN_{L_{i-1}:f_i} \oplus N_M = L\oplus T,
	\]
	so the union $Q^\prime = \tilde Q \cup \tilde Q_1 \cup \dots \cup \tilde Q_s$ is a $D$-standard cone decomposition of $T$ and HF$_T=\text{HF}_{N_L}$ is clear. Finally by Lemma \ref{5.4}, $Q^\prime$ can be refined to a $D$-exact cone decomposition $P$ of $T$. Notice that applying the two lemmas and taking union does not decrease the degree of the cone decomposition, hence $\text{deg}(Q)\leq \text{deg}(\tilde{Q})\leq \text{deg}(Q^\prime)\leq \text{deg}(P)$. 
\end{proof}

Now it remains to reduce from a vector space $T$ with HF$_T=\text{HF}_{N_{IF}}$ and $I$ generated by a regular sequence of degrees $d_1,\dots,d_{n-r}$ to the monomial submodule $(x_1^{d_1},\dots,x_{n-r}^{d_{n-r}})F$, using the fact that they have the same Hilbert function. We connect the two reductions to get our main theorem of this section. Recall that if $\mathbb{K}$ is infinite, Lemma \ref{2.6} guarantees that there exists an ideal $I\subseteq S$ generated by a regular sequence of degrees $d_1,\dots,d_{n-r}$ with $IF\subseteq M$, hence the assumption of the following theorem can be satisfied for an arbitrary module with dim$(F/M)=r$.
\begin{theorem} \label{6.4}
	Let $M\subsetneq F$ be a graded submodule generated by homogeneous elements $\{g_ie_j:i=1,\dots,n-r, j=1,\dots,m\}\cup\{f_1,\dots,f_s\}$, where $g_1,\dots,g_{n-r}\in S$ is a homogeneous regular sequence of degrees $d_1,\dots,d_{n-r}$ and $D=\text{max}\{\text{deg}(f_i):i=1,\dots,s\}\geq l$. Fix a monomial order $\prec_F$ on F, if $Q$ is a $l$-standard cone decomposition of $N_M$, then 
	\[
	1+\text{deg}(Q)\leq \text{max}\{1+\text{deg}(P^+), d_1+\dots+d_{n-r}+l-n+1 \}
	\]
	where $P$ is a $D$-exact cone decomposition of $N_{JF}$ and $J=(x_1^{d_1},\dots,x_{n-r}^{d_{n-r}})$.
\end{theorem}

\begin{proof}
	Let $I=(g_1,\dots,g_{n-r})\subseteq S$ and $L=IF\subseteq M$. By Lemma \ref{6.3}, we can complete any $l$-standard cone decomposition $Q$ of $N_M$ to a $D$-exact cone decomposition $\tilde Q$ of a vector space $T\subseteq F$ with $\text{HF}_{T}=\text{HF}_{N_{L}}$ and $\text{deg}(\tilde Q) \geq \text{deg}(Q)$. Since $g_1,\dots,g_{n-r}$ and $x_1^{d_1},\dots,x_{n-r}^{d_{n-r}}$ are both regular sequences of the same degrees, $N_{L}$ and $N_{JF}$ have the same Hilbert function by Lemma \ref{2.1}. As HF$_T=\text{HF}_{N_{L}}=\text{HF}_{N_{JF}}$, by Lemma \ref{5.6} the Macaulay constants $b_0$ of $\tilde Q$ and $P$ are the same, so $\text{deg}(\tilde{Q})=\text{deg}(P)$. Finally by Lemma \ref{6.3}, \ref{2.1}, and \ref{5.6}, $1+\text{deg}(Q)\leq 1+\text{deg}(\tilde Q) =1+\text{deg}(P)= \text{max}\{1+\text{deg}(\tilde{P}^+), d_1+\cdots + d_{n-r}+l-n+1\}$.
\end{proof}

\section{Bounding the Macaulay constants}
By Theorem \ref{4.7} and Theorem \ref{6.4}, it remains to bound the Macaulay constant $b_1$ of $N_{JF}$ where $J=(x_1^{d_1},\dots,x_{n-r}^{d_{n-r}})$. As a result of the simple structure of $JF$, using induction we can explicitly construct a $D$-exact cone decomposition of $N_{JF}$ (without using the EXACT algorithm), which will give us the bound on $b_1$ easily. This section is a generalization of \citet[\S 3.3]{mayr}, in which Mayr and Ritscher bound the Macaulay constant $b_1$ of $N_J\subseteq S$. 

Since the subspace $N_{JF}$ is independent of any monomial order on $F$, the assumption on the monomial order will be omitted in the following lemmas. 

Notice that $r=\text{dim}(S/J)=\text{dim}(F/JF)$ tells us that $b_i=D$ for all $i>r$. In particular if $r=0$ then $b_1=D$, so it suffices to bound $b_1$ for $r\geq 1$.
\begin{lemma} \label{7.1}
	Let $J=(x_1^{d_1},\dots,x_{n-r}^{d_{n-r}})\subset S$ and $b_0,\dots,b_{n+1}$ be the Macaulay constants of a $D$-exact cone decomposition $P$ of $N_{JF}$. Then
	\[
	b_{n+1}=b_n=\cdots=b_{r+1}=D.
	\]
\end{lemma}

\begin{proof}
	It follows from the fact $\text{deg}(\text{HP}_{N_{JF}})=r-1$ and Lemma \ref{5.3}.
\end{proof}

The following lemma presents the base case of the induction needed for Lemma \ref{7.5} and is obvious from the definition of $N_{JF}$.
\begin{lemma} \label{7.2}
	Let $J=(x_1^{d_1},\dots,x_{n-r}^{d_{n-r}})$ be an ideal of $S$. Then $N_{JF}$ may be decomposed as
	\[
	N_{JF}=T_r \times \mathbb{K}[x_{n-r+1},\dots,x_n]=\bigoplus_{h\in \text{Mon}(T_r)}h\mathbb{K}[x_{n-r+1},\dots,x_n],
	\] 
	where the vector space $T_r$ is given by 
\[
T_r = \text{span}_{\mathbb{K}} \{\boldsymbol{x^\alpha} e_j \in F : j=1,\dots,m,
	\boldsymbol{x^\alpha} \in \mathbb{K}[x_1,\dots,x_{n-r}], 0\leq \alpha_i<d_i \text{ for }i=1,\dots,n-r\}.
\]
\end{lemma}

When $r\geq 1$, we can compute $b_r$ using the decomposition in Lemma \ref{7.2} and the following lemma.
\begin{lemma} \label{7.3}
	Let $T_k\subseteq F$ be a finite dimensional vector space generated by monomials and $P_k$ a cone decomposition of $T_k \times \mathbb{K}[x_{n-k+1},\dots,x_n]$. Then $P_k$ has exactly dim$_{\mathbb{K}}(T_k)$ cones of dimension $k$.
\end{lemma}

\begin{proof}
	See \citet[Lemma 29]{mayr}.
\end{proof}

\begin{corollary} \label{7.4}
	Let $J=(x_1^{d_1},\dots,x_{n-r}^{d_{n-r}})$ be an ideal of $S$ and $b_0,\dots,b_{n+1}$ the Macaulay constants of a $D$-exact cone decomposition $P$ of $N_{JF}$. If $r\geq 1$, then
	\[
	b_r=d_1\cdots d_{n-r}m+D.
	\]
\end{corollary}

\begin{proof}
	Apply Lemma \ref{5.3} and the previous three lemmas to get $b_r=b_{r+1}+\text{dim}_{\mathbb{K}}T_r=D+d_1\cdots d_{n-r}m$.
\end{proof}

Now we construct a $D$-exact cone decomposition $P$ of $N_{JF}$ layer by layer. That is to say in the $k$th inductive step, we "peel off" from $T_k \times \mathbb{K}[x_{n-k+1},\dots,x_n]$ dimension $T_k$ many cones of dimension $k$, which will become all the dimension $k$ cones in $P$, and take $T_{k-1} \times \mathbb{K}[x_{n-k+2},\dots,x_n]$ to be the complement.
\begin{lemma} \label{7.5}
	Let $J=(x_1^{d_1},\dots,x_{n-r}^{d_{n-r}})$ be an ideal of $S$. Then for any $D\geq \text{max}\{2,l\}$ and $k=0,\dots,r$, there exists a $D$-exact cone decomposition $P_k$ and a finite-dimensional subspace $T_k\subseteq N_{JF}\cap \oplus_{j=1}^{m} \mathbb{K}[x_1,\dots,x_{n-k}]e_j$ which have a monomial basis such that
	\[
	N_{JF}=(T_k \times \mathbb{K}[x_{n-k+1},\dots,x_n]) \oplus \bigoplus_{C\in P_k} C.
	\]
	Let $b_0,\dots,b_{n+1}$ be the Macaulay constants of $P_0$, then $b_{k-1}\leq \frac{1}{2}b_k^2$ for $k=2,\dots,r$.
\end{lemma}

\begin{proof}
	We will first construct $P_0,\dots,P_r$ first and then bound $b_1,\dots,b_{r-1}$. Inductively, we construct $P_{k-1}\supseteq P_k$ and $T_{k-1}\times \mathbb{K}[x_{n-k+2},\dots,x_n]\subseteq T_k \times \mathbb{K}[x_{n-k+1},\dots,x_n]$, so that the following three conditions hold: 
	\begin{enumerate}
		\item $P_{k-1}\setminus P_k$ consists of dim$_{\mathbb{K}} (T_{k-1})$ cones of dimension $k$.
		\item If $\{h_1,\dots,h_s\}$ is a monomial basis of $T_{k-1}$ with $\text{deg}(h_1)\leq \cdots \leq \text{deg}(h_s)$, then for each $i$, $\text{deg}(h_{i})\leq \text{deg}(h_{i-1})+1$ whenever $\text{deg}(h_{i})\geq l+1$.
		\item $N_{JF}=(T_{k-1} \times \mathbb{K}[x_{n-k+2},\dots,x_n]) \oplus \bigoplus_{C\in P_{k-1}} C.$
	\end{enumerate}
	Notice that condition (1) implies that $b_{k-1}=b_k+\text{dim}_{\mathbb{K}} (T_{k-1})$ for $k=2,\dots,r$.
	
	The inductions starts with $k=r$. Let $P_r=\emptyset$ and $T_r$ be as in Lemma \ref{7.2}, then it is easy to see that they satisfy the three conditions. Let $1\leq k\leq r$ and assume $P_k$ and $T_k$ have been constructed, we want to construct $P_{k-1}$ and $T_{k-1}$. Notice that $P_k\subseteq P_0$ contains all cones of dimension larger than $k$, therefore $b_n,\dots,b_{k+1}$ are fixed. Let $\{h_1,\dots,h_s\}$ be a monomial basis of $T_k$ with $\text{deg}(h_1)\leq \cdots \leq \text{deg}(h_s)$ and choose
	\[
	C_i=h_ix_{n-k+1}^{b_{k+1}+i-\text{deg}(h_i)-1}\mathbb{K}[x_{n-k+1},\dots,x_n] \text{ for } i=1,\dots,s.
	\]
	Define
	\[
	\begin{split}
	T_{k-1}
	&=\text{span}_{\mathbb{K}}\{h_ix_{n-k+1}^c:i=1,\dots,s, c=0,\dots,b_{k+1}+i-\text{deg}(h_i)-2\}\\
	&\subseteq \bigoplus_{j=1}^{m}\mathbb{K}[x_1,\dots,x_{n-k+1}]e_j.
	\end{split}
	\]
	
	In order for these $C_i$'s to be well-defined, we need to show that $b_{k+1}+i-\text{deg}(h_i)-1\geq 0$. Since $b_{k+1}\geq b_{n+1}=D\geq l$, if $\text{deg}(h_i)\leq l$, then $b_{k+1}+i-\text{deg}(h_i)-1\geq l+i-l-1\geq 0$. If $\text{deg}(h_i)\geq l+1$, we induct on $i$ and use the fact that  $T_k$ satisfies condition (2) to get
	\[
	\begin{split}
	0
	&\leq b_{k+1}+(i-1)-\text{deg}(h_{i-1})-1\\
	&\leq b_{k+1}+(i-1)-(\text{deg}(h_i)-1)-1\\
	&=b_{k+1}+i-\text{deg}(h_i)-1.
	\end{split}
	\]
	
	Hence the $C_i$'s are well-defined. It is easy to see that $C_i\subseteq T_k \times \mathbb{K}[x_{n-k+1},\dots,x_n]$, deg$(C_i)=b_{k+1}+i-1$, and dim$(C_i)=k$. Thus $P_{k-1}=P_k\cup \{C_1,\dots,C_s\}$ is a $D$-exact cone decomposition. 
	
	Notice that $T_{k-1}$ satisfies condition (2) as $T_k$ does, and 
	\[
	T_k \times \mathbb{K}[x_{n-k+1},\dots,x_n] = C_1\oplus \cdots \oplus C_s \oplus (T_{k-1} \times \mathbb{K}[x_{n-k+2},\dots,x_n]),
	\]
	hence it follows by induction that
	\[
	N_{JF}=(T_{k-1} \times \mathbb{K}[x_{n-k+2},\dots,x_n]) \oplus \bigoplus_{C\in P_{k-1}} C.
	\]
	
	Inductively we have constructed $P_0,\dots,P_r$, now we turn to the computation of the Macaulay constants $b_1,\dots,b_r$ of $P_0$. Let $2\leq k\leq r$ and we want to prove $b_{k-1}\leq \frac{1}{2}b_k^2$. As $b_{k-1}=b_k+\text{dim}_{\mathbb{K}}T_{k-1}$ (from condition (1)), it suffices to bound $\text{dim}_{\mathbb{K}}T_{k-1}$. 
	By definition of $T_{k-1}$,
	\[
	\text{dim}_{\mathbb{K}}T_{k-1}=\sum_{i=1}^{s} (b_{k+1}+i-\text{deg}(h_i)-1) \leq \sum_{i=1}^{s} (b_{k+1}+i-1)=sb_{k+1}+\frac{1}{2}s(s-1).
	\]
	Since $s=\text{dim}_{\mathbb{K}}T_k=b_k-b_{k+1}$, using induction and the inequality $b_{k+1}\geq D\geq 2$ we get
	\[
	\begin{split}
	b_{k-1}=\text{dim}_{\mathbb{K}}T_{k-1}+b_k
	&\leq (b_k-b_{k+1})b_{k+1}+\frac{1}{2}(b_k-b_{k+1})(b_k-b_{k+1}-1)+b_k\\
	&=\frac{1}{2}(b_k^2-b_{k+1}^2+b_k+b_{k+1})\\
	&\leq \frac{1}{2}(b_k^2-b_{k+1}^2+\frac{1}{2}b_{k+1}^2+b_{k+1}) \leq \frac{1}{2}b_k^2
	\end{split}
	\]
\end{proof}

Notice that Lemma \ref{7.5} shows that if $\{h_1,\dots,h_s\}$ is a monomial basis of $T_0$, then $P=P_0\cup \{C(h_i,\emptyset): i=1,\dots,s\}$ is a $D$-exact cone decomposition of $N_{JF}$ with Macaulay constants $b_1,\dots,b_{n+1}$ equal to those of $P_0$. Combining Corollary \ref{7.4} and Lemma \ref{7.5}, we have
\begin{corollary} \label{7.6}
	Let $J=(x_1^{d_1},\dots,x_{n-r}^{d_{n-r}})\subseteq S$ and $b_0,\dots,b_{n+1}$ the Macaulay constants of a $D$-exact cone decomposition $P$ of $N_{JF}$ where $D\geq \text{max}\{2,l\}$. Then if $r\geq 1$,
	\[
	b_k\leq 2\left[\frac{1}{2}(d_1\cdots d_{n-r}m+D) \right]^{2^{r-k}} \text{ for } k=1,\dots,r.
	\]
\end{corollary}

Finally, we combine all the previous results to obtain a bound for the Gr\"{o}bner basis degree of an arbitrary graded module. Recall that $S=\mathbb{K}[x_1,\dots,x_n]$, $F=Se_1\oplus \cdots \oplus Se_m$ with deg$(e_j)\geq 0$ for all $j$, and $l=\text{max}\{\text{deg}(e_j):j=1,\dots,m\}$. Without loss of generality we may assume the maximum generating degree $D$ of $M$ is greater or equal to $l$, since otherwise $M$ is irrelevant to the summand with the largest degree (say $Se_1$) so we may replace $F$ by $Se_2\oplus \cdots \oplus Se_m$. 

We first prove Theorem \ref{7.7} which uses assumption on the generating degrees of the 0th Fitting ideal, and then prove Theorem \ref{7.8} which uses only the generating degree of $M$. Notice that if $M=I$ is an ideal in $S$, then as $I=\text{Fitt}_0(S/I)$, Theorem \ref{7.7} gives the same bound as the bound of Mayr and Ritscher (see Theorem \ref{1.2}).

\begin{theorem} \label{7.7}
	Let $M\subsetneq F$ be a graded submodule generated by homogeneous elements of maximum degree $D$ where $D\geq l$ and dim$(F/M)=r$. Let Fitt$_0(F/M)$ be generated by homogeneous polynomials $p_1,\dots,p_k$ of degrees $d_1\geq \cdots \geq d_k$. 
	
	If $r=0$, then the degree of the reduced Gr\"{o}bner basis $G$ of $M$ for any monomial order on $F$ is bounded by
	\[
	\text{deg}(G)\leq d_1+\cdots+d_n+l-n+1.
	\]
	
	If $r\geq 1$, then the degree of the reduced Gr\"{o}bner basis $G$ of $M$ for any monomial order on $F$ is bounded by
	\[
	\text{deg}(G)\leq 2\left[\frac{1}{2}(d_1\cdots d_{n-r}m+D) \right]^{2^{r-1}}.
	\]
\end{theorem}

\begin{proof}
	Without loss of generality we may assume $\mathbb{K}$ is infinite. If $r=0$, notice that $d_1+\cdots+d_n+l-n+1\geq D\geq l$, so we are done by Theorem \ref{4.7}, Theorem \ref{6.4}, and Lemma \ref{7.1}.
	
	Assume $r\geq 1$. Let $f_1,\dots,f_s$ be a generating set of $M$ with $D=\text{max}\{\text{deg}(f_i):i=1,\dots,s\}$. We may assume $D\geq 2$, since otherwise it is easy to see that the Gr\"{o}bner basis degree of $M$ is bounded by $1$. 
	By Lemma \ref{2.6}, Fitt$_0(F/M)\subseteq M:_S F$ contains a regular sequence $g_1,\dots,g_{n-r}$ of degrees $d_1\geq \cdots \geq d_{n-r}$. Then $g_ie_j\in M$ for $i=1,\dots,n-r$ and $j=1,\dots, m$, and so $\{g_ie_j:i=1,\dots,n-r, j=1,\dots,m\}\cup\{f_1,\dots,f_s\}$ is a generating set of $M$. Let $J=(x_1^{d_1},\dots,x_{n-r}^{d_{n-r}})$ and let $P$ be a $D$-exact cone decomposition of $N_{JF}$ with Macaulay constants $b_0,\dots,b_{n+1}$. By Corollary \ref{7.6}, $1+\text{deg}(P^{+})=b_1$ is bounded by $2\left[\frac{1}{2}(d_1\cdots d_{n-r}m+D) \right]^{2^{r-1}}$ which is greater than $d_1+\cdots+d_{n-r}+l-n+1$ and $l$. Hence by Theorem \ref{4.7} and Theorem \ref{6.4}, the reduced Gr\"{o}bner basis degree is bounded by 
	\[
	\begin{split}
	\text{deg}(G)
	&\leq \text{max}\{1+\text{deg}(P^{+}), d_1+\cdots+d_{n-r}+l-n+1, l\}\\
	&\leq 2\left[\frac{1}{2}(d_1\cdots d_{n-r}m+D) \right]^{2^{r-1}}.
	\end{split}
	\]
\end{proof}

\begin{theorem} \label{7.8}
	Let $M\subsetneq F$ be a graded submodule generated by homogeneous elements $f_1,\dots,f_s$ with degrees $D_1\geq \cdots \geq D_s$, $D=D_1\geq l$, and dim$(F/M)=r$. 
	
	If $r=0$, then the degree of the reduced Gr\"{o}bner basis $G$ of $M$ for any monomial order on $F$ is bounded by
	\[
	\text{deg}(G) \leq \left(D_1+\cdots+D_m-\sum_{j=1}^{m}\text{deg}(e_j)\right)n+l-n+1 \leq Dmn-n+1.
	\]
	
	If $r\geq 1$, then the degree of the reduced Gr\"{o}bner basis $G$ of $M$ for any monomial order on $F$ is bounded by
	\[
	\begin{split}
	deg(G)
	&\leq 2\left[\frac{1}{2}((D_1+\cdots+D_m-\sum_{j=1}^{m}\text{deg}(e_j))^{n-r}m+D) \right]^{2^{r-1}}\\
	&\leq 2\left[\frac{1}{2}\left((Dm)^{n-r}m+D\right) \right]^{2^{r-1}}.
	\end{split}
	\]
\end{theorem}

\begin{proof}
	Choose a minimal homogeneous generating set of Fitt$_0(M)$ with degrees $d_1\geq \cdots \geq d_{k}$ and use Lemma \ref{2.7} to bound $d_1,\dots,d_{n-r}$, then apply Theorem \ref{7.7}.
\end{proof}

A bound only depending on $n$, $m$, and $D$ can be easily deduced from Theorem \ref{7.8}. If $r=n$, then deg$(G)\leq 2\left[\frac{1}{2}(m+D) \right]^{2^{n-1}}$. If $r\leq n-1$, the bound decreases when $r$ decreases, so deg$(G)\leq 2\left[\frac{1}{2}(Dm^2+D) \right]^{2^{n-2}}$. Picking a bound that is greater than both, we have:
\begin{corollary} \label{7.9}
	Let $M\subsetneq F$ be a graded submodule generated by homogeneous elements with maximum degree $D\geq l$. Then the degree of the reduced Gr\"{o}bner basis $G$ for any monomial order on $F$ is bounded by
	\[
	\text{deg}(G)\leq 2(Dm)^{2^{n-1}}.
	\]
\end{corollary}

\section{Non-graded case}

To solve the non-graded case, it seems natural to homogenize $M$ using an additional variable $t$ and deduce the non-graded bound from the graded bound simply by replacing $n$ with $n+1$. However this approach is false for a dimension-dependent bound since homogenizing an ideal may increase the dimension (see \citet[Example 3.4]{mayr}). This problem can be fixed by the following lemma which follows from \citet[Lemma 3.15]{sombra} and the proof of \citet[Theorem 4.19]{binaei}

\begin{lemma}[\cite{binaei}] \label{8.2}
	Let $\mathbb{K}$ be an infinite field and $I\subsetneq S=\mathbb{K}[x_1,\dots,x_{n}]$ be an ideal with dim$(S/I)=r$ generated by polynomials $p_1,\dots,p_k$ of degrees $d_1\geq \cdots \geq d_k$. Then there are polynomials $g_1,\dots,g_{n-r}\in I$ such that $g_1^h,\dots,g_{n-r}^h\in I^h$ form a regular sequence and deg$(g_i)\leq d_1\cdots d_{n-r}$ for $i=1,\dots,n-r$.
\end{lemma}

We will apply Lemma \ref{8.2} to the ideal Fitt$_0(F/M)$ to get a regular sequence $g_1^h,\dots,g_{n-r}^h\in \text{Fitt}_0(F/M)^h$. If $M$ is generated by elements $f_1,\dots,f_s$, we consider the graded module $\tilde{M}\subseteq M^h$ generated by $\{g_i^he_j:i=1,\dots,n-r, j=1,\dots,m\}\cup\{f_1^h,\dots,f_s^h\}$. Then since the dehomogenization of a Gr\"{o}bner basis of $\tilde{M}$ is a Gr\"{o}bner basis of $M$, it suffices to bound the Gr\"{o}bner basis degree of $\tilde{M}$. Now all of our previous tools can be applied as $\tilde{M}$ is graded. 

Notice that if $M=I$ is an ideal in $S$, then as $I=\text{Fitt}_0(S/I)$, Theorem \ref{8.3} gives a bound that is sharper than Mayr and Ritscher's bound (see Theorem \ref{1.2}).

Recall that $F$ is a free module over $\mathbb{K}[x_1,\dots,x_n]$ with basis elements $e_1,\dots,e_m$ with deg$(e_j)\geq 0$ for all $j$ and $l=\text{max}\{\text{deg}(e_j):j=1,\dots, m\}$.
\begin{theorem} \label{8.3}
	Let $M\subsetneq F$ be a submodule generated by elements of maximum degree $D$ with $D\geq l$ and dim$(F/M)=r$. Let Fitt$_0(F/M)$ be generated by polynomials $p_1,\dots,p_k$ of degrees $d_1\geq \cdots \geq d_k$. Then the degree of the reduced Gr\"{o}bner basis $G$ for any monomial order on $F$ is bounded by
	\[
	\text{deg}(G) \leq 2\left[\frac{1}{2}\left((d_1\cdots d_{n-r})^{n-r}m+D\right) \right]^{2^{r}}.
	\]
\end{theorem}

\begin{proof}
	Without loss of generality we may assume $\mathbb{K}$ is infinite. Fix a monomial order $\prec$ on $F$ and let $\prec^\prime$ be its extension on $F^h$ (defined in Section 2.5). Let $M$ be generated by $f_1,\dots,f_s\in F$ with $D=\text{max}\{\text{deg}(f_i):i=1,\dots,s\}$, and let Fitt$_0(F/M)$ be generated by $p_1,\dots,p_k\in S$ of degrees $d_1\geq \cdots \geq d_k$. By Lemma \ref{8.2}, there exists polynomials $g_1,\dots,g_{n-r}\in \text{Fitt}_0(F/M)$ with deg$(g_i):=\tilde{d}_i\leq d_1\cdots d_{n-r}$ for $i=1,\dots,n-r$, and $g_1^h,\dots,g_{n-r}^h$ form a regular sequence. Consider the graded module $\tilde{M}\subsetneq F^h$ generated by $\{g_i^he_j:i=1,\dots,n-r, j=1,\dots,m\}\cup\{f_1^h,\dots,f_s^h\}$, and let $\tilde{G}$ be a reduced Gr\"{o}bner basis of $\tilde{M}$ w.r.t. $\prec^\prime$. Notice that we have the inclusions $\oplus_{i=1}^{s} S[t]f_i^h\subseteq \tilde{M}\subseteq M^h$. By Lemma \ref{2.9}, $\tilde{G}^{deh}$ is a Gr\"{o}bner basis of $M$ whose degree is clearly bounded by the degree of $\tilde{G}$, so it suffices to bound $deg(\tilde{G})$.
	
	By Theorem \ref{4.7}, there exists a $l$-standard cone decomposition $Q$ of $N_{\tilde{M}}$ with deg$(\tilde{G})\leq \text{max}\{1+\text{deg}(Q),l\}$. Let $J=(x_1^{\tilde{d_1}},\dots,x_{(n+1)-(r+1)}^{\tilde{d}_{(n+1)-(r+1)}})$, then by Theorem \ref{6.4}, $1+\text{deg}(Q)\leq \text{max}\{b_1,\tilde{d}_1+\cdots+\tilde{d}_{n-r}+l-(n+1)+1\}$ where $b_1$ is the Macaulay constant of a $D$-exact cone decomposition $P$ of $N_{JF}$. As $r+1\geq 1$, by Corollary \ref{7.6} $b_1$ is bounded by $2\left[\frac{1}{2}(\tilde{d}_1 \cdots \tilde{d}_{n-r}m+D) \right]^{2^{r}}$, which is greater than $\tilde{d}_1+\cdots+\tilde{d}_{n-r}+l-(n+1)+1$ and $l$.
	
	Finally we combine the above inequalities to get
	\[
	\begin{split}
	\text{deg}(G)\leq \text{deg}(\tilde{G})
	&\leq \text{max}\{b_1,\tilde{d}_1+\cdots+\tilde{d}_{n-r}+l-(n+1)+1,l\}\\
	&\leq 2\left[\frac{1}{2}(\tilde{d}_1\cdots \tilde{d}_{n-r}m+D) \right]^{2^{r}}\\
	&\leq 2\left[\frac{1}{2}((d_1\cdots d_{n-r})^{n-r}m+D) \right]^{2^{r}}.
	\end{split}
	\]
\end{proof}

\begin{theorem}
	Let $M\subsetneq F$ be a submodule generated by elements of maximum degree $D$ with $D\geq l$ and $dim F/M=r$. Then the degree of the reduced Gr\"{o}bner basis $G$ for any monomial order on $F$ is bounded by
	\[
	\text{deg} (G) 
	\leq 2\left[\frac{1}{2}\left(\left( Dm  \right)^{(n-r)^{2}}m+D\right) \right]^{2^{r}}.
	\]
\end{theorem}

\begin{proof}
	Choose a minimal generating set of Fitt$_0(M)$ with degrees $d_1\geq \cdots \geq d_k$ and use Lemma \ref{2.7} to bound $d_1,\dots, d_{n-r}$, then apply Theorem \ref{8.3}.
\end{proof}

To get a bound that does not depend on the dimension, we replace $n$ by $n+1$ in the bound given by Corollary \ref{7.9}. 
\begin{corollary}
	Let $M\subsetneq F$ be a submodule generated by elements of maximum degree $D$ with $D\geq l$. Then the degree of the reduced Gr\"{o}bner basis $G$ of $M$ for any monomial order on $F$ is bounded by
	\[
	\text{deg}(G)\leq 2(Dm)^{2^{n}}.
	\]
\end{corollary}

\section*{Acknowledgement}

The author would like to thank her advisor Giulio Caviglia for proposing this problem.

\bibliographystyle{elsarticle-harv}
\bibliography{Module_Grobner_basis_bound}{}

\end{document}